\documentclass[12pt]{amsart}

\usepackage[margin=1in,letterpaper,portrait]{geometry}
\usepackage{amsmath}
\usepackage{amssymb}
\usepackage{amsthm}
\usepackage{mathrsfs}
\usepackage{verbatim}
\usepackage{fancyhdr}
\usepackage{url}
\usepackage{tikz}
\usepackage{enumerate}
\usetikzlibrary{matrix,arrows}
\usepackage{tikz-cd}
\usepackage[pdftex,bookmarksnumbered=true,linkbordercolor={1 0 0}]{hyperref}
\usetikzlibrary{plotmarks}
\usepackage{pgfplots}
\usepgfplotslibrary{fillbetween}

\usepackage{setspace}

\newcommand{\mb}{\mathbb}

\newcommand{\mc}{\mathcal}
\newcommand{\im}{{\rm im}}
\newcommand{\ol}{\overline}
\newcommand{\wt}{\widetilde}
\newcommand{\codim}{{\rm codim}}
\newcommand{\Hilb}{{\rm Hilb}}
\newcommand{\RHilb}{\wt{{\rm Hilb}}}

\newcommand*{\sheafhom}{\mathscr{H}\kern -.5pt om}
\renewcommand{\O}{\mathscr{O}}
\newcommand*{\on}[1]{\operatorname{#1}}
\newcommand{\M}{\ol{\mc{M}}}
\newcommand{\ev}{\on{ev}}
\newcommand{\Ve}{\on{Vertex}}
\newcommand{\Ed}{\on{Edge}}
\newcommand{\Fl}{\on{Flag}}
\newcommand{\Ta}{\on{Tail}}

\begin{document}
\theoremstyle{plain}
\newtheorem{Thm}{Theorem}[section]
\newtheorem{Cor}[Thm]{Corollary}
\newtheorem{Conj}[Thm]{Conjecture}
\newtheorem{Pro}[Thm]{Problem}
\newtheorem{Main}{Main Theorem}
\renewcommand{\theMain}{}
\newcommand{\Sol}{\text{Sol}}
\newtheorem{Lem}[Thm]{Lemma}
\newtheorem{Claim}[Thm]{Claim}
\newtheorem{Prop}[Thm]{Proposition}
\newtheorem{Exam}{Example}
\newtheorem{ToDo}{To Do}

\theoremstyle{definition}
\newtheorem{Def}[Thm]{Definition}
\newtheorem{Exer}[Thm]{Exercise}
\newtheorem{Rem}[Thm]{Remark}

\theoremstyle{remark}

\title{A note on rational curves on general Fano hypersurfaces}
\author{Dennis Tseng}
\address{Dennis Tseng, Harvard University, Cambridge, MA 02138}
\email{DennisCTseng@gmail.com}
\date{\today}

\allowdisplaybreaks

\begin{abstract}
We show the Kontsevich space of rational curves of degree at most roughly $\frac{2-\sqrt{2}}{2}n$ on a general hypersurface $X\subset \mb{P}^n$ of degree $n-1$  is equidimensional of expected dimension and has two components: one consisting generically of smooth, embedded rational curves and the other consisting of multiple covers of a line. This proves more cases of a conjecture of Coskun, Harris, and Starr and shows the Gromov-Witten invariants in these cases are enumerative.
\end{abstract}

\maketitle

\section{Introduction}
Our investigation is motivated by the following conjecture by Coskun, Harris, and Starr.
\begin{Conj}
[{\cite[Conjectures 1.3 and 1.4]{CS}}]
\label{CHS}
Let $X\subset \mb{P}^n$ be a general hypersurface of degree $d\leq n$ and dimension at least 3. Then, the open locus $R_{e}(X)$ in the Hilbert scheme of $X$ parameterizing smooth rational curves of degree $e$ is irreducible of dimension $e(n+1-d)+n-4$. Furthermore, if $d\leq n-1$, then the evaluation map $\M_{0,1}(X,e)\to X$ is flat. 
\end{Conj} 

There has been progress on Conjecture \ref{CHS} by using induction on $e$ via bend and break \cite{HRS, BM2, RY}. Most recently, Riedl and Yang showed that Conjecture \ref{CHS} holds when $d\leq n-2$ \cite{RY}. For work on rational curves on arbitrary smooth hypersurfaces, see \cite{CS, BV}. 

In this note, we will look at rational curves of low degree on general hypersurfaces of degree $n-1$ in over an algebraically field of characteristic zero. Specifically, we will show
\begin{Thm}
\label{RC2} 
If $X\subset \mb{P}^{n}$ is a general hypersurface of degree $n-1$ for $n\geq 4$, and
\begin{align*}
e<n-\frac{1+\sqrt{n^2-n-15}}{2},
\end{align*}
then the evaluation map $\M_{0,1}(X,e)\to X$ is flat and the Kontsevich space $\M_{0,0}(X,e)$ is a local complete intersection stack of pure dimension $2e+n-4$ and has two components. One of the components consists of $e$ to 1 covers of a line and the other component consists generically of smooth rational curves. 
\end{Thm}
Another method that has been successful in controlling the dimensions of $R_e(X)$ has been to consider the incidence correspondence between hypersurfaces and rational curves in projective space. In the setting of Conjecture \ref{CHS}, $R_e(X)$ is smooth and of the expected dimension for $e\leq d+2$ by work of Gruson, Lazarsfeld and Peskine \cite{GLP}. Furukawa made this connection explicit and gave a weaker bound that also works in positive characteristic \cite{F}. 

An advantage in considering the Kontsevich space $\M_{0,0}(X,e)$ instead of $R_e(X)$ is the connection with Gromov-Witten invariants. The space $\M_{0,0}(X,e)$ is given by the zero section of a vector bundle on $\M_{0,0}(\mb{P}^n,e)$ and the virtual fundamental class of $\M_{0,0}(X,e)$ is the Euler class \cite[26.1.3]{Hori}. In the cases covered by Theorem \ref{RC2}, the fundamental class of $\M_{0,0}(X,e)$ agrees with the virtual fundamental class as the dimensions agree. Furthermore, it is possible to see that each component of $\M_{0,0}(X,e)$ is generically smooth, as we will identify smooth points parameterizing highly reducible curves in every component in our analysis. In particular by Kleiman-Bertini, given subvarieties $V_1,\ldots, V_r$ of $\mb{P}^n$, then the number of degree $e$ rational curves in $X$ meeting each of $V_1,\ldots,V_r$ is 
\begin{align}
\int_{[\M_{0,r}(X,e)]^{{\rm vir}}}{\ev_1^{*}[V_1']\cap \cdots \cap \ev_r^{*}[V_r']} \label{GWeq}
\end{align}
provided the integrand in \eqref{GWeq} is expected dimension zero and $V_i'$ is a general $PGL_{r+1}$ translate of $V_i$. 

Finally, we remark that hypersurfaces of degree $n-1$ in $\mb{P}^n$ are examples of Fano varieties of pseudo-index 2.  Given a Fano variety $M$, the pseudo-index of $M$ is defined as
\begin{align*}
 \min\{-K_M\cdot C\mid \ C\text{ is a rational curve of }M\}.
\end{align*}
If $f: C\to M$ is a rational curve on $M$ with $-K_M\cdot C=2$ with $f_{*}[C]=\beta$, then the expected dimension of $\M_{0,0}(M,r\beta)$ is the same as the expected dimension of $r$-fold covers of stable maps in $\M_{0,0}(M,\beta)$. From this point of view, it is clear why the two components given in Theorem \ref{RC2} are necessary. See \cite{Castravet} for another example of a Fano variety of pseudoindex 2 and a description of the components of its space of rational curves. 

\subsection{Methods}
In order to apply bend and break in families as in \cite{RY}, we will need to apply a couple of results from the author \cite{T} to control the locus of hypersurfaces with more lines through a point than expected and the locus of hypersurfaces with positive dimensional singular loci. Also, the space of lines through a point is expected to be finite, and so in particular is not irreducible. This will present a technical obstacle in showing irreducibility of the main component in Theorem \ref{RC2}. To deal with this, we will show that the space of conics through a general point is irreducible, and then use an argument that is similar in spirit to the irreducibility argument in \cite{HRS} but will require us to specialize further. We work in characteristic zero, but it seems likely that the techniques extend to positive characteristic. 

\subsection{Outline}
The argument will have two parts. In Section \ref{Flat}, we will show the fibers of the evaluation map $\M_{0,1}(X,e)\to X$ have the expected dimension in Theorem \ref{FEV}. We will look at the irreducible components of the general fiber in Theorem \ref{fiber} in Section \ref{Irreducible}. Theorem \ref{RC2} will follow from Theorems \ref{FEV} and \ref{fiber}. 

In Section \ref{Flat}, we are mostly interested in dimension, so it suffices to work with the coarse moduli space of $\M_{0,1}(X,e)$, but we will need look at smoothness in Section \ref{Irreducible}, so we will need to work with $\M_{0,1}(X,e)$ as a stack. In general, the fact that the Kontsevich space $\M_{0,1}(X,e)$ and, more generally, the Behrend-Manin stacks \cite{BM} are Deligne-Mumford stacks in characteristic zero will allow us to avoid technical difficulties with stacks by passing to an \'etale cover. 

\section{Acknowledgements}
The author would like to thank Jason Starr for help in identifying the irreducible components in Theorem \ref{RC2} and Eric Riedl and David Yang for helpful conversations. The author would also like to thank the referee for detailed suggestions on how to improve the exposition and for the additional references to the literature. 

\section{Conventions}
Throughout the paper, we will work over an algebraically closed field of characteristic zero. We will let
\begin{enumerate}
\item
$X\subset \mb{P}^n$ be a hypersurface of degree $2\leq d\leq n-1$
\item
$N=\binom{d+n}{n}-1$ so $\mb{P}^N$ parameterizes hypersurfaces of degree $d$ in $\mb{P}^n$
\item
$\mc{X}\to \mb{P}^N$ be the universal hypersurface $\mc{X}=\{(p,X)|p\in X\}\subset \mb{P}^n\times \mb{P}^N$
\item
$\M_{0,r}(X,e)$ is the Kontsevich space parameterizing stable maps $C\to X$ of degree $e$, where $C$ is a genus $0$ curve with $r$ marked points
\item
$\M_{0,r}(\mc{X}/\mb{P}^N,e)$ is the relative Kontsevich space parameterizing stable maps mapping into the fibers of $\mc{X}\to \mb{P}^N$.
\end{enumerate}
We are interested in when $d=n-1$. 

\section{Flatness}
\label{Flat}
The goal of the first half of the note is to prove
\begin{Thm}
\label{FEV}
If $e<n-\frac{1+\sqrt{n^2-n-15}}{2}$, then for a general hypersurface $X$ of degree $d=n-1$ in $\mb{P}^n$, the evaluation map
\begin{align*}
\M_{0,1}(X,e)\to X
\end{align*}
is flat and $\M_{0,1}(X,e)$ is a local complete intersection stack. 
\end{Thm}

\subsection{Definitions}

We recall the notion of $e$-level \cite{RY}, with some slight modifications. 
\begin{Def}
Given a point $p\in X\subset \mb{P}^n$ of a hypersurface of degree $d$, $p$ is called $e$-\emph{level} if one of the following holds:
\begin{enumerate}
\item
$p$ is a smooth point and the space of degree $e$ rational curves through $p$ has dimension at most $e(n-d+1)-2$
\item
$p$ is a singular point and the space of degree $e$ rational curves through $p$ has dimension at most $e(n-d+1)-1$.
\end{enumerate}
\end{Def}

By the space of rational cuves through $p$, we mean the fiber of the evaluation map $\M_{0,1}(X,e)\to X$ over $p\in X$. When $p$ fails to be $e$-level, the space of rational curves through $p$ has larger dimension than expected. The notion of $e$-levelness extends the notion of flatness of the evaluation map to singular points in a manner that allows us bound the locus of points that are not $e$-level. 

\begin{Def}
A hypersurface $X$ is $e$-\emph{level} if,
\begin{enumerate}
\item
the singular locus is finite, and
\item
every point of $X$ is $k$-level for all $k\leq e$. 
\end{enumerate}
\end{Def}
\begin{Rem}
\label{lremark}
Instead of requiring only finitely many singular points, what is actually being used in the dimension counts of \cite[Proposition 2.5]{CS} \cite[Proposition 5.5]{RY} is that there is no rational curve $C\subset X$ of degree less than $e$, for which the space of rational curves of degree $k<e$ through every point $p\in C$ has dimension exceeding $k(n-d+1)-2$. For example, the original definition of $e$-level \cite{RY} replaced the condition of finitely many singular points with the condition that there is no rational curve of degree at most $e$ in the singular locus. 
\end{Rem}

\subsection{Basic lemmas}
We collect here some crucial facts needed to run the argument. 

\begin{Prop}
\label{CD1}
For any map $\phi: T\to \M_{0,r}(X,e)$ from an irreducible scheme $T$, the pullback $\phi^{-1}(\partial \M_{0,r}(X,e))\subset T$ is empty or has codimension at most 1. 
\end{Prop}

\begin{proof}
This follows from the fact that $\partial\M_{0,0}(\mb{P}^n,e)$ is a divisor in $\M_{0,0}(\mb{P}^n,e)$ \cite[Theorem 3]{FP}. 
\end{proof}

We will also need a version of bend and break. 
\begin{Lem} \cite[Corollary 3.3]{RY}
\label{BB}
If $T$ is a closed locus in $\M_{0,0}(\mb{P}^n,e)$ of dimension at least $2n-1$, then $T$ contains maps with reducible domains. 
\end{Lem}

\subsection{Codimension of the locus of hypersurfaces that are not 1-level}

As described in \cite{RY}, the idea of the argument is to borrow rational curves from nearby hypersurfaces to apply bend and break. To run the argument, we need to know the locus of hypersurfaces that are not $1$-level has high codimension. For the rest of this section, let $\mb{P}^N$ be the space of all hypersurfaces of degree $d$ in $\mb{P}^n$. We are primarily interested in the case where $d=n-1$. 

\begin{Thm}
[{\cite{T}}] 
\label{LT}
Let $U\subset \mb{P}^N$ be the open locus of smooth hypersurfaces. For $4\leq d=n-1$, a largest component of the closed locus $Z\subset U$ of hypersurfaces that are not 1-level consists of hypersurfaces containing a 2-plane. 
\end{Thm}
We will not need this but the largest component is unique when $n-1=d>4$. 

We also need to consider hypersurfaces with a larger dimensional family of lines through a singular point than expected.
\begin{Prop}
\label{SS}
Let $U\subset \mb{P}^N$ be the open locus of hypersurfaces with at most finitely many singular points. Let $Z\subset \mb{P}^N$ be the locus of hypersurfaces $X$ for which there exists a singular point $p$ containing an $(n-d+1)$-dimensional family of lines in $X$. For $d\leq n-1$, the codimension of $Z$ in $\mb{P}^N$ is at least $\binom{n+1}{2}$. 
\end{Prop}
\begin{proof}
This is proven in \cite[Proposition 5.10]{RY}. Even though they assume $d\leq n-2$, the analysis for the case of a singular point goes through. The main obstruction to extending \cite[Proposition 5.10]{RY} to the case where $d=n-1$ (case of smooth points) is covered by Theorem \ref{LT}. 
\end{proof}

To prove a hypersurface is $e$-level, we need to rule out a positive dimensional singular locus. 
\begin{Thm}
[{\cite{T}}] 
\label{KT}
For $d\geq 7$, the unique largest component of the closed locus $Z\subset \mb{P}^N$ of hypersurfaces with positive dimensional singular locus consists of hypersurfaces singular along a line. 
\end{Thm}

\begin{Cor}
\label{S1}
Suppose $7\leq d=n-1$. Let $Z\subset \mb{P}^N$ be the locus of hypersurfaces that are not 1-level. Then, the codimension of $Z$ is $\binom{n+1}{2}-3(n-2)$. 
\end{Cor}

\begin{proof}
We need to compare the contributions of hypersurfaces that
\begin{enumerate}
\item
have a larger dimensional family of lines than expected through a smooth point,
\item
have a larger dimensional family of lines than expected through a singular point,
\item
or have a positive-dimensional singular locus.
\end{enumerate}
The first case happens in codimension $\binom{n+1}{2}-3(n-2)$ by Theorem \ref{LT}. The dimension of hypersurfaces singular along a line is $dn-2n+3=n^2-3n+3$ \cite[Lemma 5.1]{K} and bounds the dimension of hypersurfaces with positive-dimensional singular locus by Theorem \ref{KT}. 
The third case is covered by Proposition \ref{SS}. Comparing these bounds yields the result. 
\end{proof} 

\subsection{Bend and break in families}
Let $\mb{P}^N$ be the space of degree $d$ hypersurfaces in $\mb{P}^n$, and $S_e\subset \mb{P}^N$ denote the closure of the hypersurfaces that are not $e$-level. We have a chain of inclusions
\begin{align*}
S_1\subset S_2\subset S_3\subset \cdots\subset \mb{P}^N.
\end{align*}
We will use the argument in \cite[Theorem 6.2]{RY} to bound the codimension of each inclusion $S_e\subset S_{e+1}$. 
\begin{Lem}
\label{OL}
Suppose $A\subset B\times C$ with projections $\pi_1: A\to B$ and $\pi_2:A\to C$. If $C$ is a projective scheme and $H\subset C$ is a general linear section. Then, $\dim(\pi_1(\pi_2^{-1}(H)))=\dim(\pi_1(A))$ if $\pi_1: A\to B$ has positive-dimensional fibers, respectively $\dim(\pi_1(\pi_2^{-1}(H)))=\dim(\pi_1(A))-1$ if $\pi_1: A\to B$ is generically finite onto its image. 
\end{Lem}

\begin{proof}
If $A\to B$ has positive dimensional fibers, choose $H$ so that it cuts down the dimension of a general fiber of $\pi_1$ by 1. 
\end{proof}

\begin{Thm}
\label{BBF}
The codimension of $S_{e-1}\subset S_{e}$ is at most $2n-(n-d+1)e$. 
\end{Thm}
\begin{proof}
Let $a=e(n-d+1)-2$ be the expected dimension of a fiber $\ev: \M_{0,1}(\mc{X}/\mb{P}^N,e)\to \mc{X}$. 
\begin{center}
\begin{tikzcd}
\M_{0,1}(\mc{X}/\mb{P}^N,e) \arrow[r] \arrow[bend left=20, rr] \arrow[d]  \arrow[bend left=20, dd] & \M_{0,1}(\mb{P}^n,e) \arrow[r] & \M_{0,0,}(\mb{P}^n,e)\\
\mc{X} \arrow[d] & &\\
\mb{P}^N & &
\end{tikzcd}
\end{center}
Suppose $\mc{A}\subset \M_{0,1}(\mc{X}/\mb{P}^N,e)$ is an irreducible component of 
\begin{align*}
\{[C] \in  \M_{0,1}(\mc{X}/\mb{P}^N,e)|  \ \ev([C])\text{ is a singular point}, \text{dimension of fiber of $\ev$ at }[C]\geq a+2\}. 
\end{align*}
If $\mc{A}$ consists of covers of lines, then $\mc{A}\to \mb{P}^N$ maps into $S_1$. If $\mc{A}$ contains a map that is not a cover of a line, then the image of $\mc{A}\to \M_{0,0}(\mb{P}^n,e)$ has dimension at least $3n-3$ from $PGL_{n+1}$-invariance, as we can interpolate a curve from $\mc{A}$ through $3$ general points by taking the $PGL$-translates of a single curve. (It is easy to do a bit better, for example by considering the dimension of the space of conics, but it is only important to us that the dimension of the image of $\mc{A}\to \M_{0,0}(\mb{P}^n,e)$ is at least $2n-1$.)

Now, we apply Lemma \ref{OL} to the image of $\mc{A}$ under the forgetful map 
\begin{align*}
\mc{A}\subset \M_{0,1}(\mc{X}/\mb{P}^N,e)\to \M_{0,0}(\mc{X}/\mb{P}^N,e)\subset \mb{P}^N\times \M_{0,0}(\mb{P}^n,e).
\end{align*}
to cut $\mc{A}$ by hyperplane sections in $\mb{P}^N$ to produce $\mc{A}'\subset \mc{A}$ such that $\mc{A}'\to \M_{0,0}(\mb{P}^n,e)$ has image dimension $2n-1$ and $\im(\mc{A}'\to  \M_{0,0}(\mc{X}/\mb{P}^N,e))\to \M_{0,0}(\mb{P}^n,e)$ is generically finite onto its image, so $\im(\mc{A}'\to  \M_{0,0}(\mc{X}/\mb{P}^N,e))$ is also dimension $2n-1$. By Lemma \ref{BB}, $\mc{A}'$ contains maps from reducible curves. The image of $\mc{A}'\to \M_{0,0}(\mc{X}/\mb{P}^N,e)\to \mb{P}^N$ has dimension at most $2n-1-(a+2)$. 

Finally, we assume, for the sake of contradiction, that the codimension of $S_{e-1}$ in the image of $\mc{A}\to \mb{P}^N$ is at least $2n-(n-d+1)e+1=2n-(a+2)+1$. Then, by construction, $\mc{A}'\to \mb{P}^N$ misses the locus $S_{e-1}$ completely. Applying \cite[Proposition 5.5]{RY} shows the locus of reducible curves in $\mc{A}'$ has dimension at most $e(n-d+1)-2=a$ in each fiber of the evaluation map $\mc{A}'\to \mc{X}$. This implies the locus of reducible curves in $\mc{A}'$ has codimension at least 2, contradicting Proposition \ref{CD1}. 

Similarly, to finish we need to consider the case where a smooth point is not level. We let $\mc{A}$ be an irreducible component of the closure of
\begin{align*}
\{[C] \in  \M_{0,1}(\mc{X}/\mb{P}^N,e)|  \ \ev([C])\text{ is a smooth point}, \text{dimension of fiber of $\ev$ at }[C]\geq a+1\}. 
\end{align*}
If $\mc{A}$ consists of covers of lines, then $\mc{A}\to \mb{P}^N$ maps into $S_1$. Otherwise, as before, $\mc{A}\to \M_{0,0}(\mb{P}^n,e)$ has dimension at least $3n-3$. Let $\mc{A}$ by hyperplane sections in $\mb{P}^N$ to produce $\mc{A}'\subset \mc{A}$ such that $\mc{A}'\to \M_{0,0}(\mb{P}^n,e)$ has image dimension $2n-1$ and the image of $\mc{A}'\to \mb{P}^N$ is at most $2n-1-(a+1)$. As before, Lemma \ref{BB} shows $\mc{A}'$ contains reducible curves. 

If we assume, for the sake of contradiction, that the codimension of $S_{e-1}$ in the image of $\mc{A}\to \mb{P}^N$ is at least $2n-(n-d+1)e+1=2n-(a+1)+1$, then $\mc{A}'\to \mb{P}^N$ misses the locus $S_{e-1}$ completely. Applying \cite[Proposition 5.5]{RY} shows the locus of reducible curves in $\mc{A}'$ has dimension at most $e(n-d+1)-2=a$ in each fiber of the evaluation map $\mc{A}'\to \mc{X}$ over a singular point and has dimension at most $a-1$ in each fiber over a smooth point. Since the general point in the image of $\mc{A}'\to \mc{X}$ is smooth, again we have the locus of reducible curves in $\mc{A}'$ has codimension at least 2, contradicting Proposition \ref{CD1}. 
\end{proof}

\subsection{Conclusion of argument}
\label{conclusiondimension}

\begin{proof} (of Theorem \ref{FEV})
To prove Theorem \ref{FEV}, it suffices to show that the fibers of the evaluation map are of the expected dimension \cite[Lemma 4.5]{HRS}. 

First suppose $d\geq 7$. We apply Theorem \ref{BBF} to show $S_e$ is not all of $\mb{P}^N$. By Corollary \ref{S1}, we have $S_1$ has codimension at least $\binom{n+1}{2}-3(n-2)$ in $\mb{P}^N$. Theorem \ref{BBF} implies then that $S_e$ has codimension at least 
\begin{align*}
\binom{n+1}{2}-3(n-2)-\sum_{e'=2}^{e}{(2n-2e')}&=\binom{n+1}{2}-3(n-2)-2n(e-1)+e(e+1)-2\\
&=\frac{1}{2}(n^2-n-4ne+2e^2+2e+8).
\end{align*}
Solving for the values of $e$ for which $n^2-n-4ne+2e^2+2e+8>0$ yields Theorem \ref{FEV} in the case $d\geq 7$. 

If $d<7$, then $e\leq 2$. We can assume $d\in \{5,6\}$ \cite[Theorem 1.6]{CS}. As mentioned in Remark \ref{lremark}, we chose to restrict to hypersurfaces with finitely many singular points to simplify the argument for $d\geq 7$. When $e=2$, we can replace $S_1$ with the locus of hypersurfaces for which all the points are 1-level and the singular locus does not contain a line. Hypersurfaces singular along a line have codimension $n^2-3n+3$ \cite[Lemma 5.1]{K}, which is at least the codimension of hypersurfaces containing a $2$-plane for $n\geq 3$. Then, we can run the same argument in Proposition \ref{BBF} to see flatness in Theorem \ref{FEV} as \cite[Proposition 5.5]{RY} still applies. 
\end{proof}

\section{Irreducible components}
\label{Irreducible}

\subsection{Conics through a point}
We will work with the Hilbert scheme of conics instead of the Kontsevich space to focus on degree 2 maps that are not covers of a line. 
\begin{Def}
We let 
\begin{enumerate}
\item
$\Hilb_{2t+1}(\mb{P}^n)$ denote the Hilbert scheme of conics in $\mb{P}^n$
\item
$\Hilb_{2t+1}(X)$ denote the Hilbert scheme of conics in $X$
\item
$\Hilb_{2t+1}(\mc{X}/\mb{P}^N)$ denote the relative Hilbert scheme of conics in the fibers of $\mc{X}\to \mb{P}^N$
\item
$\mc{C}\to \Hilb_{2t+1}(\mc{X}/\mb{P}^N)$ be the universal curve. 
\end{enumerate}
\end{Def}
Note that $\Hilb_{2t+1}(\mb{P}^n)$ is smooth, as a $\mb{P}^5$ bundle over $\mb{G}(2,n)$, $\Hilb_{2t+1}(\mc{X}/\mb{P}^N)$ is smooth as $\Hilb_{2t+1}(\mc{X}/\mb{P}^N)\to \Hilb_{2t+1}(\mb{P}^n)$ is a $\mb{P}^{N-5}$-bundle, and $\mc{C}$ is smooth as $\mc{C}\to \mb{P}^n$ is a $\mb{G}(1,n-1)\times \mb{P}^4\times \mb{P}^{N-5}$ bundle. 

The goal of this section is to prove
\begin{Prop}
\label{conic}
The general fiber of $\mc{C}\to \mc{X}$ is smooth and connected. 
\end{Prop}

The proof of \cite[Theorem V.4.3]{KJ} regarding lines on a general hypersurface is very similar, and so is good preparation for the proof of Proposition \ref{conic}. The proof of Proposition \ref{conic} will reduce to Propositions \ref{secant} and \ref{codim2} below. 

\begin{Def}
For $a>0$, let $W_a=H^0(\mb{P}^1,\O_{\mb{P}^1}(a))$ be the degree $a$ forms in $2$ variables.
\end{Def}

The case $a=1$ of Lemma \ref{determinental} below is given in \cite[Lemma V.4.3.1.1]{KJ}. 
\begin{Lem}
\label{determinental}
Consider the multiplication map $m: W_a\times W_{b}\to W_{a+b}$. For a linear hyperplane $V\subset W_{a+b}$, let 
\begin{align*}
D(V):=\{f\in W_b\mid m(W_a\times \{f\})\subset V\}\subset W_b.
\end{align*}
Then, $\codim(D(V))=\min\{a,b,\ell\}$, where $\ell$ is the smallest positive integer such that $V\in \mb{P}(W_{a+b})^{*}$ lies on the $\ell$-secant variety $S_{\ell-1}(C)$,. Here $C\subset \mb{P}(W_{a+b})^{*}$ is the rational normal curve given as the image $\mb{P}^1\to \mb{P}(W_{a+b})^{*}$ where $p\in \mb{P}^1$ maps to $\{A\in W_{a+b}\mid A(p)=0\}$. 
\end{Lem}

\begin{proof}
Let the elements of $W_b$, and $W_{a+b}$ be written as $b_0s^b+b_1s^{b-1}t+\cdots+b_b t^b$, and $c_0s^{a+b}+c_1s^{a+b-1}t+\cdots+c_{a+b}t^{a+b}$ respectively. So $f= b_0s^b+b_1s^{b-1}t+\cdots+b_b t^b$ is in $D(V)$ if and only if $ b_0s^{b+i}t^{j}+b_1s^{b+i-1}t^{j+1}+\cdots+b_b s^{i}t^{b+j}\in V$ for all $i,j\geq 0$ with $i+j=a$. If $V$ is written as defined by $d_0c_0+\cdots+d_{a+b}c_{a+b}=0$, then $f\in D(V)$ if and only if
\begin{align}
\label{matrixeq}
\begin{pmatrix}
c_0 & c_1 & \cdots & c_{b}\\
c_1 & c_2 & \cdots & c_{b+1}\\
\vdots & \vdots & \ddots & \vdots \\
c_a & c_{a+1} & \cdots& c_{a+b}
\end{pmatrix}
\begin{pmatrix}
b_0 \\ b_1\\ \vdots \\ b_b
\end{pmatrix}
&= 
\begin{pmatrix}
0 \\ 0 \\ \vdots \\ 0
\end{pmatrix}
\end{align}
Therefore, $\codim(D(V))$ is the rank of the matrix in \eqref{matrixeq} with $i,j$ entry $c_{i+j}$ for $0\leq i\leq a$, $0\leq j\leq b$. By \cite[Proposition 9.7]{H}, the locus in $\mb{P}(W_{a+b})^{*}$ where the matrix is rank $\ell$ is given by the $\ell$-secant variety $S_{\ell-1}(C)$ of the rational normal curve $C\subset \mb{P}(W_{a+b})^{*}$ given by the rank 1 matrices. Therefore, $C$ is given as the image of $\mb{P}^1\to \mb{P}(W_{a+b})^{*}$ mapping $[s:t]$ to $[s^{a+b}: s^{a+b-1}t: \cdots : t^{a+b}]$. So $[s:t]$ maps to the hyperplane in $\mb{P}(W_{a+b})$ defined by $c_0s^{a+b}+c_1s^{a+b-1}t+\cdots+c_{a+b}t^{a+b}$, which is precisely $\{A\in W_{a+b}\mid A(p)=0\}$ for $p=[s:t]$. 
\end{proof}

\begin{Lem}
\label{multiplication}
Suppose $n\geq 4$ and $d<n$. Then, the closed locus $Z$ of forms $(G,F_3,\ldots,F_n)$ for which the map $W_{3}\times W_{1}^{n-2}\to W_{2d-1}$ given by $(A,L_3,\ldots,L_n)\mapsto AG+L_3F_3+\cdots+L_nF_n$ is not surjective has codimension at least 2.
\end{Lem}

\begin{proof}
Consider the incidence correspondence 
$$\Phi=\{(V,G,F_3,\ldots,F_n)\mid GW_3+F_3W_1+\cdots+F_nW_1\subset V\}\subset \mb{P}(W_{2d-1})^{*}\times W_{3}\times W_{1}^{n-2}\to W_{2d-1}.$$ 

It suffices to show $\dim(\Phi)\leq \dim(W_3\times W_1^{n-2})$. To do so, we will consider the projection $\pi: \Phi\to \mb{P}(W_{2d-1})^{*}$, stratify $\mb{P}(W_{2d-1})^{*}$ and bound the loci in $\Phi$ lying over the strata separately using Lemma \ref{determinental}. As in Lemma \ref{determinental}, let $E\subset \mb{P}(W_{2d-1})^{*}$ be the rational curve parameterizing hyperplanes in $\mb{P}(W_{2d-1})$ of the form $\{A\in W_{a+b}\mid A(p)=0\}$ and let $E=S_1(E)\subset S_2(E)\subset S_3(E)\subset \mb{P}W_{2d-1}^{*}$ its first, second and third secant varieties. Now, we have four cases:
\begin{enumerate}
\item
The locus $S_1(E)$ is 1-dimensional. For $V\in S_1(E)$, $\pi^{-1}(V)$ is codimension $n-1$ in $W_{3}\times W_{1}^{n-2}$ by Lemma \ref{determinental}. Therefore, 
\begin{align*}
\dim(W_{3}\times W_{1}^{n-2})-\dim(\pi^{-1}(S_1(E)))=n-2. 
\end{align*}
\item
The locus $S_2(E)$ is 3-dimensional. For $V\in S_2(E)\backslash S_1(E)$, $\pi^{-1}(V)$ is codimension $2(n-1)$ in $W_{3}\times W_{1}^{n-2}$ by Lemma \ref{determinental}. Therefore, 
\begin{align*}
\dim(W_{3}\times W_{1}^{n-2})-\dim(\pi^{-1}(S_2(E)\backslash S_1(E))=2n-5. 
\end{align*}
\item
The locus $S_3(E)$ is at most 5-dimensional. For $V\in S_3(E)\backslash S_2(E)$, $\pi^{-1}(V)$ is codimension $2(n-1)+1$ in $W_{3}\times W_{1}^{n-2}$ by Lemma \ref{determinental}. Therefore, 
\begin{align*}
\dim(W_{3}\times W_{1}^{n-2})-\dim(\pi^{-1}(S_3(E)\backslash S_2(E))=2n-6.
\end{align*}
If $d=3$ and $n=4$, then 
\item
For $V\in \mb{P}(W_{2d-1})^{*}\backslash S_3(E)$,  $\pi^{-1}(V)$ is codimension $2(n-1)+2$ in $W_{3}\times W_{1}^{n-2}$ by Lemma \ref{determinental}. Therefore, 
\begin{align*}
\dim(W_{3}\times W_{1}^{n-2})-\dim(\pi^{-1}(\mb{P}(W_{2d-1})^{*}\backslash S_3(E))=2n-2d+1.
\end{align*}
\end{enumerate}
Since $n\geq 4$ and $d<n$, each case yields a bound that is at least $2$, which is what we wanted. 

\end{proof}

\begin{Prop}
\label{secant}
Let $\RHilb_{2t+1}(\mc{X}/\mb{P}^N)\subset \Hilb_{2t+1}(\mc{X}/\mb{P}^N)$ denote the open locus of smooth conics. The singular locus of $\mc{C}\to \mc{X}$ has codimension at least 2 for $n\geq 4$. 
\end{Prop}

\begin{proof}
Since all smooth conics are projectively equivalent, we can fix the smooth conic $C\subset \mb{P}^n$ with ideal $(Q(X_0,X_1,X_2),X_3,\cdots,X_n)$ and parameterization $\mb{P}^1\to \mb{P}^n$ given by $(s,t)\to (s^2,st,t^2,0,\ldots,0)$. It suffices to show that, in the $\mb{P}^{N-5}$ hypersurfaces $X$ that contain $C$, the locus of hypersurfaces $X$ with $h^1(N_{C/X}(-p))\neq 0$ has codimension at least $2$ for $p\in C$. 

Let $F(X_0,\ldots,X_n)$ cut out a hypersurface $X$ containing $C$. Then, $F$ can be written as
\begin{align*}
F &= Q(X_0,X_1,X_2)G(X_0,X_1,X_2)+X_3F_3(X_0,\ldots,X_n)+\cdots+X_nF_n(X_0,\ldots,X_n),
\end{align*}
where $G$ is degree $d-2$ and each $F_i$ is degree $d-1$. The polynomials $G$ and $F_i$ can be chosen independently. 
Consider the short exact sequence
\begin{center}
\begin{tikzcd}
0 \arrow[r] & N_{C/X}  \arrow[r] & N_{C/\mb{P}^n} \arrow[r] \arrow[equal,d]  &N_{X/\mb{P}^n}|_{C} \arrow[r] \arrow[equal,d]   &0\\
& & \O_C(2H)\oplus \O_C(H)^{\oplus (n-2)} \arrow[r] & \O_C(dH) &
\end{tikzcd}
\end{center}
To determine the map $\O_C(2H)\oplus \O_C(H)^{\oplus (n-2)} \to \O_C(dH)$, we consider the dual map. We have the conormal bundles $N_{C/\mb{P}^n}^\vee=\frac{(Q,X_3,\ldots,X_n)}{(Q,X_3,\ldots,X_n)^2}$ and $N_{X/\mb{P}^n}^{\vee}|_{C}=\frac{(F)}{(F)(Q,X_3,\ldots,X_n)}$ as quotients of ideal sheaves. The map $N_{X/\mb{P}^n}^\vee|_{C}\to N_{C/\mb{P}^n}^\vee$ is induced by inclusion $(F)\subset (Q,X_3,\ldots,X_n)$. Dualizing, the map $\O_C(2H)\oplus \O_C(H)^{\oplus (n-2)}\to \O_C(dH)$ is given by multiplication by the vector $(G,F_3,\ldots,F_n)$. The long exact sequence in cohomology implies $H^1(N_{C/X}(-p))=0$ if and only if $H^0(\O_C(2H-p)\oplus \O_C(H)^{\oplus (n-2)})\to H^0(\O_C(dH))$ is surjective. 

By pulling back via the parameterization $\mb{P}^1\to \mb{P}^n$, we are done by Lemma \ref{multiplication}. 
\end{proof}

\begin{Prop}
\label{codim2}
For $n\geq 4$, the singular locus of $\mc{C}\to \mc{X}$ has codimension at least 2 in $\mc{C}$. 
\end{Prop}

\begin{proof}
The space $\Hilb_{2t+1}(\mb{P}^n)$ can be stratified into three loci: smooth conics, unions of two distinct lines, and doubled lines. These strata are of codimensions $0$, $1$ and $2$, respectively. Since the Hilbert function is constant on $\Hilb_{2t+1}(\mb{P}^n)$, the map $\mc{C}\to \Hilb_{2t+1}(\mb{P}^n)$ is a Zariski-local projective bundle. This means the pullbacks of these strata in $\mc{C}$ are also of codimensions $0$, $1$, and $2$. Therefore, it suffices to show Proposition \ref{codim2} when restricted to the smooth conics and to the unions of two distinct lines. 

By Proposition \ref{secant}, we know Proposition \ref{codim2} is true when we restrict $\mc{C}$ to the locus where $\mc{C}\to \Hilb_{2t+1}(\mc{X}/\mb{P}^N)$ is smooth. Now, consider the locally closed subset $\mc{Y}\subset \mc{C}$ consisting of pointed curves $(C,p)$, where $C$ is a union of two distinct lines $L_1,L_2$, $p\in L_1\backslash L_2$. The locus $\mc{Y}$ is irreducible, as we can parameterize $\mc{Y}$ by a smooth, irreducible variety $\mc{Z}\to \mc{Y}$ by specifying $(C,p)$ by first choosing $p\in \mb{P}^n$, the two plane $P$ containing $C$, a line $L_1$ containing $p$ and contained in $P$, a second line $L_2$ contained in $P$, and finally a hypersurface $X$ containing $L_1\cup L_2$. Since the complement of the union of $\mc{Y}$ and the smooth locus of $\mc{C}\to \Hilb_{2t+1}(\mc{X}/\mb{P}^N)$ has codimension 2, it suffices to show the singular locus of $\mc{C}\to \mc{X}$ has codimension 1 in $\mc{Y}$. 

A parameter count involving the normal bundle of a line in a hypersurface similar to \cite[Proposition V.4.3.9]{KJ} shows the singular locus of the map from the universal line $\M_{0,1}(\mc{X}/\mb{P}^N,1)\to \mc{X}$ is singular in codimension at least $\min\{n-2,n-d\}=n-d$. A more complicated version of this parameter count was given in detail in the proof of Proposition \ref{secant} above. This means the map $\mc{Z}\to \mc{X}$ is smooth at a general point. Since the image $\mc{Y}$ of $\mc{Z}\to \mc{C}$ has codimension $1$ and $\mc{Z}\to \mc{Y}$ has finite reduced fibers, this means $\mc{C}\to \mc{X}$ is smooth at a general point of $\mc{Y}$. 
\end{proof}

\begin{proof}[Proof of Proposition \ref{conic}]
By \cite[Proposition 3.1]{Minoccheri}, it suffices to see $\mc{X}$ is \'etale simply-connected and the singular locus of $\mc{C}\to \mc{X}$ has codimension at least 2. The first condition follows from the fact that $\mc{X}\to \mb{P}^n$ is rational as a Zariski-locally trivial $\mb{P}^{N-1}$ bundle and the birational invariance of \'etale fundamental groups. 
The second condition follows Proposition \ref{codim2} below. 
\end{proof}

\subsection{Layeredness}
Let $X\subset \mb{P}^n$ be a hypersurface of degree $d$, where $d\leq n-1$. 

\begin{Def}
A point $p\in X$ is called $e$-\emph{layered} if it is 1-level and for every $1<k\leq e$, every irreducible component parameterizing degree $k$ rational curves through $p$ contains reducibles. 
\end{Def}

\begin{Def}
A hypersurface $X$ is $e$-layered if it is $e$-level and a general point is $e$-layered.
\end{Def}

\begin{Rem}
\label{layeredremark}
We have corresponded with the authors of \cite{RY} and it is not clear how their argument as written shows $e$-layeredness at all points as claimed, but they brought us to the attention that $e$-levelness at all points and $e$-layeredness at a general point suffices to prove their main theorem. 
\end{Rem}

\subsection{Existence of $e$-layered hypersurfaces}
Let $\mb{P}^N$ be the space of degree $d$ hypersurfaces in $\mb{P}^n$, and $T_e\subset \mb{P}^N$ denote the closure of the hypersurfaces that are not $e$-layered. We have a chain of inclusions
\begin{align*}
T_1\subset T_2\subset T_3\subset \cdots\subset \mb{P}^N.
\end{align*}
\begin{Thm}
\label{BBF2}
The codimension of $T_{e-1}\subset T_{e}$ is at most $2n-(n-d+1)e$. 
\end{Thm}

\begin{Prop}
\label{FP}
If $f:X\to Y$ is a map between Noetherian schemes of finite presentation and $Z\subset X$ is a closed subset, then the set $A\subset X$ of all $x\in X$ such that a geometric component of the fiber $f^{-1}(f(x))$ containing $x$ is disjoint from $Z$ is constructible. 
\end{Prop}

\begin{proof}
Since this is a statement about the underlying topological spaces, we can assume $X$ and $Y$ are reduced. By restricting to each component of $X$, we can assume $X$ is integral. By Noetherian induction on $X$, it suffices to prove Proposition \ref{FP} after restriction to some open subset of $Y$, so we can assume $Y={\rm Spec}(A)$ is affine and integral. Let $\eta\in Y$ be the generic point and $X_{\eta}$ be the fiber over the generic point. Following \cite[Tag 0551]{stacks-project}, we will find a finitely presented surjection $Y'\to Y$ such that each component of the generic fiber of $X\times_{Y}Y'$ is irreducible as follows. Since $X_{\eta}$ has finitely many components, there is a finite separable extension $L$ of $K(A)$ such that each component of $X_{\eta}\times _{{\rm Spec}(K(A))} {\rm Spec}(L)$ is geometrically irreducible \cite[EGA IV, Theorem 4.4.4, Proposition 4.5.9]{EGAIV2}. 

Since $L$ is separable over $K(A)$, it is generated by some element $\alpha$, which we can assume to be in $A$ after multiplying by an element of $A$. Then, if the field extension $L$ over $K(A)$ is defined by the monic polynomial $p$, then let $A'=A[T]/p(T)$ and $Y'={\rm Spec}(A')$. 
If we prove Proposition \ref{FP} for the morphism $X\times_{Y}Y'\to Y'$, then we prove Proposition \ref{FP} via Chevellay's Theorem \cite[EGA IV, Theorem 1.8.4]{EGAIV1} applied to $X\times_{Y}Y'\to X$. Therefore, we can assume in addition that $X_{\eta}$ has geometrically irreducible components.  



Let $X_{1,\eta},\ldots,X_{r,\eta}$ denote the irreducible components of $X_{\eta}$ and $X_1,\ldots,X_r$ the closures of $X_{1,\eta},\ldots,X_{r,\eta}$ in $X$ respectively. Since $X\backslash (X_{1,\eta}\cup \cdots \cup X_{r,\eta})\to Y$ misses $\eta\in Y$, Chevellay's Theorem again implies we can restrict $Y$ to a standard affine that avoids the image. Therefore, we can assume  $X$ is the set-theoretically the union of $X_1,\ldots,X_r$. 

By generic flatness \cite[EGA IV, Theorem 6.9]{EGAIV2} and the fact that flat morphisms are open \cite[EGA IV, Theorem 2.4.6]{EGAIV2}, we can replace $Y$ by an open subset so that each $X_i\to Y$ is flat and surjective. 

By generic flatness again, we can replace $Y$ with a standard open to assume $X_i\cap X_j\to Y$ is flat for each pair $1\leq i,j\leq r$. By equidimensional of the fibers of a flat morphism \cite[EGA IV, Corollary 6.1.4]{EGAIV2}, we know $f^{-1}(p)\cap (X_i\cap X_j)\subset f^{-1}(p)\cap X_i$ is nowhere dense for $i\neq j$.  

By replacing $Y$ with an open subset, we can assume that each geometric fiber of $X_i\to Y$ is irreducible \cite[EGA IV, Theorem 9.7.7(i)]{EGAIV3}. Finally, if we let $S\subset \{1,\ldots,r\}$ be the subset of indices $i$ such that $X_i$ does not intersect $Z$, then $A$ is the union $\bigcup_{i\in S}{X_i}$. 
\end{proof}

\begin{proof}[Proof of Theorem \ref{BBF2}]
From Theorem \ref{BBF}, we know the codimension of $T_{e-1}\subset T_{e-1}\cup S_e$ is at most $2n-(n-d+1)e$. Using the same setup as the proof of Theorem \ref{BBF}, let $\mc{B}$ be \begin{align*}
\{[C]\in \M_{0,1}(\mc{X}/\mb{P}^N,e)|\text{ a component of }\ev^{-1}(\ev([C]))\text{ containing }[C]\text{ has no reducibles}\}.
\end{align*}
Note that $\mc{B}$ is constructible by Proposition \ref{FP}, so in particular every component of $\ol{\mc{B}}$ contains an open dense subset contained in $\mc{B}$. 

Consider the map $\ol{\mc{B}}\to \mc{X}\to  \mb{P}^N$, and consider the closed locus in $\im(\ol{\mc{B}}\to \mc{X})$ where the fiber dimension is $n-1$. Let  $\mc{B}'\subset \ol{\mc{B}}$ be the inverse image of that closed locus in $\ol{\mc{B}}$. Hence, $\mc{B}'$ are the stable maps in $\ol{\mc{B}}$ lying above the hypersurfaces $[X]\in \mb{P}^N$ that have a component covered by $\ol{\mc{B}}\to \mc{X}$ under the evaluation map. 

We want to control the image $\mc{B}'\to \mb{P}^N$. Let $\mc{A}\subset \mc{B}'$ be an irreducible component. If $\mc{A}$ contains a map that is not a cover of a line, then the image of $\mc{A}\to \M_{0,0}(\mb{P}^n,e)$ has dimension at least $3n-3$. Applying Lemma \ref{OL} as before allows us to cut $\mc{A}$ by hyperplane sections in $\mb{P}^N$ to obtain $\mc{A}'$ such that the image of $\mc{A}'\to \M_{0,0}(\mb{P}^n,e)$ has dimension $2n-1$. The image of the generic fiber of $\mc{A}'\to \M_{0,0}(\mb{P}^n,e)$ in $\mb{P}^N$ is finite. Lemma \ref{BB} shows $\mc{A}'$ contains stable maps from reducible curves. 

This means $\im(\mc{A}'\to \M_{0,0}(\mc{X}/\mb{P}^N,e))$ is also dimension $2n-1$. Since the fiber dimension of $\im(\mc{A}'\to \mc{X})\to \mb{P}^N$ is $n-1$, the fiber dimension of $\im(\mc{A}'\to \M_{0,0}(\mc{X}/\mb{P}^N,e))\to \mb{P}^N$ is at least $(n-1)+(a-1)$, where $a=(n-d+1)e-2$ as in the proof of Theorem \ref{BBF}. The image of $\mc{A}'\to \mb{P}^N$ has dimension at most $2n-1-(n+a-2)$. 

Assume, for the sake of contradiction, that the codimension of $T_{e-1}\subset T_e$ is at least $2n-(n-d+1)e+1$. Then, since $n\geq 3$ the image $\mc{A}'\to \mb{P}^N$ misses the locus $S_{e-1}$, which contradicts \cite[Proposition 5.5]{RY} and Proposition \ref{CD1}. Here, the point is \cite[Proposition 5.5]{RY} shows the locus parameterizing reducible curves intersects every fiber of $\mc{A}'\to \mc{X}$ in codimension at least 1, the general fiber of $\mc{A}'\to \mc{X}$ parameterizes no reducible curves, and singular points occur in the image of $\mc{A}'\to \mc{X}$ in codimension at least $n-1$ by the definition of $(e-1)$-levelness. This means the locus in $\mc{A}'$ parameterizing reducible curves is codimension at least $2$, which contradicts Proposition \ref{CD1}. 
\end{proof}

\begin{Cor}
\label{elayerexist}
If $e<n-\frac{1+\sqrt{n^2-n-15}}{2}$, then a general hypersurface $X$ of degree $d=n-1$ in $\mb{P}^n$ is $e$-layered. 
\end{Cor}

\begin{proof}
If $d\geq 7$, then this follows from the same dimension computation in Section \ref{conclusiondimension} applied in the proof of Theorem \ref{FEV} using Theorems \ref{BBF} and \ref{BBF2}. When $d\leq 6$, it suffices to consider the case $e\leq 2$, in which case Proposition \ref{conic} suffices.  
\end{proof}


\subsection{Behrend-Manin stacks}
Let $X\subset \mb{P}^n$ be a smooth hypersurface of degree $d$. In order to keep track of the combinatorics of the components of reducible rational curves, we will use Behrend-Manin stacks. We refer the reader to \cite[Definitions 1.6, 3.13]{BM} for the precise definitions of a stable $A$-graph $\tau$ and the associated Behrend-Manin stack $\M(X,\tau)$. Also see \cite{HRS} for a shorter account that suffices for our purposes. 

Roughly, a stable $A$-graph keeps track of the combinatorics of the irreducible components of a stable map $C\to X$, including the dual graph of how they intersect, the marked points on each component, the degree of the map restricted to each component, and the genus of each component. Since we are dealing with rational curves, all the stable $A$-graphs we consider will have genus zero, meaning the genus of each vertex is zero and the underlying graph is a tree. 

Associated to a stable $A$-graph $\tau$, there is a set of vertices $\Ve(\tau)$, a set of edges $\Ed(\tau)$ connecting them, and a set of tails $\Ta(\tau)$, which can be thought of half edges attached to vertices. There is also a map $\beta: \Ve(\tau)\to \mb{Z}_{\geq 0}$, assigning a degree to each vertex. We also let $\beta(\tau)=\sum_{v\in \Ve(\tau)}\beta(v)$ and the expected dimension
\begin{align*}
\dim(X,\tau):=(n+1-d)\beta(\tau)+\#\Ta(\tau)-\#\Ed(\tau)+\dim(X)-3 
\end{align*}
\cite[Definition 3.4]{HRS}. Finally, there is a set of flags $\Fl(\tau)$, where we have two flags corresponding to each edge in $\Ed(\tau)$, corresponding to the two endpoints, and one flag for each tail. In particular, $\#\Fl(\tau)=2\#\Ed(\tau)+\#\Ta(\tau)$. 

The Behrend-Manin stack $\M(X,\tau)$ parameterizes stable maps $C\to X$, where the curve $C$ consists of \emph{prestable curves} $C_v$ \cite[Definition 2.1]{BM}, one for each vertex of $\tau$, that glue together and map to $X$ according to the data in $\tau$. The open locus of $\M(X,\tau)$ of strict maps is quicker to define. See \cite[Definition 3.7]{HRS} for more details. 

\begin{Def}
A stable map $C\to X$ in $\M(X,\tau)$ is a \emph{strict map} if $C_v\cong \mb{P}^1$ for each $v\in \Ve(\tau)$. The locus of strict maps is an open substack $\mc{M}(X,\tau)\subset \M(X,\tau)$. 
\end{Def}

A point in $\M(X,\tau)\subset \M(X,\tau)$ can be specified by the data $((C_v)_{v\in \Ve(\tau)},(h_v:C_v\to X)_{v\in \Ve(\tau)},(q_f)_{f\in \Fl(\tau)})$ such that $q_f\in C_v$, where $v$ is the vertex to which $q_f$ is attached. Each map $h_v: C_v\to X$ is specified to have degree $\beta(v)$. 

Since we want to think of $\Ta(\tau)$ as parameterizing marked points, for each $f\in \Ta(\tau)$, we have an evaluation map 
\begin{align*}
\ev_f: \M(X,\tau)\to X
\end{align*}
\cite[Definition 3.11]{HRS}. Similarly, we have $\ev_f$ for all $f\in \Fl(\tau)$, as the remaining flags correspond to the points of intersection between different prestable curves $C_v$ that piece together to give the domain of a stable map $C\to X$, and we can ask for the image of such an intersection point. 

\subsection{A criterion for smoothness}
\begin{Def}
Let $\tau_r(e)$ be the stable $A$-graph that has one vertex $v$, no edges, $r$ tails such that $\beta(\tau)=\beta(v)=e$. By definition, $\M(X,\tau_r(e))$ is the Kontsevich space $\M_{0,r}(X,e)$. Given any $A$-stable graph with $\beta(\tau)=e$ and $\#\Ta(\tau)=r$, there is a morphism $\M(X,\tau)\to \M_{0,r}(X,e)$ canonical up to relabeling the tails. 
\end{Def}

If we repeatedly specialize a rational curve $C\to X$ so that it breaks up into more and more components, then we eventually end up with a tree of lines. Since we will care about rational curves through a general point $p\in X$ given by a tree of lines, we make the following definition. By abuse of notation, it is different than the one given in \cite[Definition 5.8]{HRS}. 

\begin{Def}
Let a stable $A$-graph $\tau$ be called a \emph{basic} $A$-graph if $\beta(v)\in \{0,1\}$ for all $v\in \Ve(\tau)$ and $\# \Ta(\tau)=1$.
\end{Def}

\begin{Def}
Let a basic $A$-graph be called \emph{nondegenerate} if $\beta(v)=1$ for all $v\in \Ve(\tau)$. 
\end{Def}

The argument in \cite[Proposition 6.6]{HRS} applied in our case gives 
\begin{Prop}
\label{smooth}
Let $\tau$ be a basic $A$-graph and $X\subset \mb{P}^n$ be a smooth $e$-level hypersurface of degree $d=n-1$ with an irreducible Fano scheme of lines. Then, the morphism $\M(X,\tau)\to \M_{0,1}(X,e)$ maps a general point of $\M(X,\tau)$ to a point in the smooth locus of $\ev: \M_{0,1}(X,e)\to X$. 
\end{Prop}

\begin{proof}
Applying the argument in \cite[Proposition 6.6]{HRS}, reduces the question of checking whether a point $(h: C\to X)\in \M(X,\tau)$ is a smooth point of $\M_{0,1}(X,e)\to X$ to checking whether $H^1(C,h^{*}T_X(-p))=0$, where $p\in C$ corresponds to the unique tail in $\Ta(\tau)$. 

The tangent space to a fiber of $\M_{0,1}(X,1)\to X$ at a pair $(\ell,p)$, where $p\in \ell\subset X$ and $\ell$ is a line is $H^0(N_{\ell/X}(-p))$, and this is of the expected dimension if and only if $H^1(N_{\ell/X}(-p))=0$. By generic smoothness, this holds for a general pair $(\ell,p)$. From the short exact sequence, $0\to T\ell\to TX|_{\ell}\to N_{\ell/X}\to 0$, $H^1(TX|_{\ell}(-p))=0$ for a general point $(\ell,p)\in \M_{0,1}(X,1)$. Applying \cite[Lemma 6.2]{HRS} allows us to conclude. 
\end{proof}

\begin{Rem}
Instead of arguing via the smoothness of the nonseparated Artin stack of prestable curves as in \cite[Proposition 6.6]{HRS} in the beginning of the proof of Proposition \ref{smooth}, an equivalent way is to add $c$ marked points to $C\to X$ so the prestable curve $C$ is actually stable. Then, it suffices to show smoothness of $\M_{0,1+c}(X,e)\to \ol{M}_{0,1+c}\times X$ at $C\to X$ as this implies the map $\M_{0,1+c}(X,e)\to X$ is smooth at $C\to X$. 

Note that $e$-levelness guarantees flatness of the evaluation map \cite[Lemma 4.5]{HRS}, so smoothness at a point is equivalent to being smooth in its fiber. 
\end{Rem}
The condition on the Fano scheme is automatically satisfied in our case since the Fano scheme of lines is smooth and connected for a general hypersurface \cite[Theorem 4.3]{KJ} if the degree $d$ of $X\subset \mb{P}^n$ is at most $2n-4$ and $X$ is not a quadric surface.  

\subsection{Rational curves through a point}
Theorem \ref{RC2} will follow from Theorem \ref{FEV} for the statement on dimension, and Corollary \ref{elayerexist} and Theorem \ref{fiber} for the statement on irreducible components. 
\begin{Thm}
\label{fiber}
Let $e\geq 2$ and $d=n-1$. If there exist $e$-layered hypersurfaces, then for a general hypersurface $X$, the fiber $F_p$ of the evaluation map
\begin{align*}
\M_{0,1}(X,e)\to X 
\end{align*}
over a general point $p\in X$ has only one component that contains curves $C\to X$ that are not multiple covers of a line. 
\end{Thm}

A diagram of the proof of Theorem \ref{fiber} in the case $e=3$ is depicted in Figure \ref{pagefigure}. 

\begin{proof}
The case $e=2$ is Proposition \ref{conic}, so suppose $e>2$. We will use strong induction on $e$. By $e$-levelness, each component of $F_p$ has the same dimension. If $C\to X$ is a rational curve in $X$ through $p$, we can use $e$-layeredness to specialize $C\to X$ to $C_0\to X$, so that $C_0\to X$ lies in $\mc{M}(X,\tau)$, where $\tau$ is a nondegenerate basic $A$-graph. By Proposition \ref{smooth}, we can assume $C_0\to X$ is a smooth point of $F_p$. 


Each component of the fiber of $\M(X,\tau)\to X$ over $p$ lies in a unique component of $F_p$. What we need to show is that as we vary over all nondegenerate basic $A$-graphs $\tau$ we only get one component of $F_p$ that contains curves that are not covers of a line. To do this, we will reduce ourselves to looking at ``combs'' of lines, where the backbone gets collapsed. One can get this by specializing a tree of lines to a ``broom'', where all the lines pass through $p$. For clarity, we will instead first reduce to the case of chains of lines and then specialize the chain of lines to a comb. 

As before, let $C_0\to X$ be in $\mc{M}(X,\tau)$, where $\tau$ is a nondegenerate basic $A$-graph. Let $(C_v)_{v\in \Ve(\tau)}$ be the components of $C_0$. Note that each $C_v\cong \mb{P}^1$. Let $v_0$ be the vertex to which the unique tail of $\tau$ is attached. By abuse of notation, we call the marked point in $C_{v_0}$ that maps to $p$ under $C\to X$ also as $p\in C_{v_0}$. Let $q_{a_1},\ldots,q_{a_r}\in C_{v_0}$ correspond to the edges attached to $v_0$ in $\tau$, or the points of attachment of the other components of $C$ to $C_{v_0}$. 
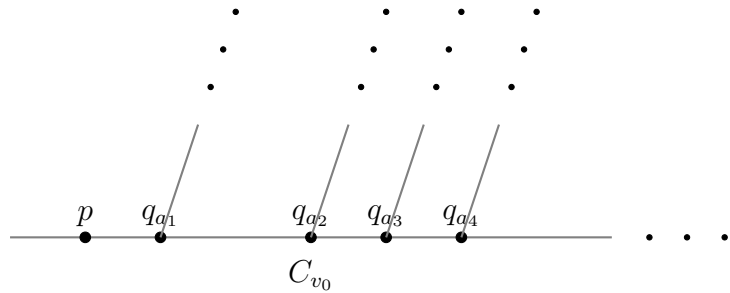
\begin{center}
\begin{tikzpicture}
\draw[gray, thick] (0,0) -- (8,0);
\filldraw[black] (1,0) circle (2pt) node[anchor=south]{$p$};
\filldraw[black] (2,0) circle (2pt) node[anchor=south]{$q_{a_1}$};
\filldraw[black] (4,0) circle (2pt) node[anchor=south]{$q_{a_2}$};
\filldraw[black] (5,0) circle (2pt) node[anchor=south]{$q_{a_3}$};
\filldraw[black] (6,0) circle (2pt) node[anchor=south]{$q_{a_4}$};
\filldraw[black] (8.5,0) circle (1pt) node{};
\filldraw[black] (9,0) circle (1pt) node{};
\filldraw[black] (9.5,0) circle (1pt) node{};
\filldraw[black] (4,-.5) circle (0pt) node{$C_{v_0}$};

\draw[gray, thick] (2,0) -- (2+.5,0+3/2);
\filldraw[black] (2+.5+1/6,3/2+1/2) circle (1pt) node{};
\filldraw[black] (2+.5+2/6,3/2+2/2) circle (1pt) node{};
\filldraw[black] (2+.5+3/6,3/2+3/2) circle (1pt) node{};

\draw[gray, thick] (4,0) -- (4+.5,0+3/2);
\filldraw[black] (4+.5+1/6,3/2+1/2) circle (1pt) node{};
\filldraw[black] (4+.5+2/6,3/2+2/2) circle (1pt) node{};
\filldraw[black] (4+.5+3/6,3/2+3/2) circle (1pt) node{};

\draw[gray, thick] (5,0) -- (5+.5,0+3/2);
\filldraw[black] (5+.5+1/6,3/2+1/2) circle (1pt) node{};
\filldraw[black] (5+.5+2/6,3/2+2/2) circle (1pt) node{};
\filldraw[black] (5+.5+3/6,3/2+3/2) circle (1pt) node{};

\draw[gray, thick] (6,0) -- (6+.5,0+3/2);
\filldraw[black] (6+.5+1/6,3/2+1/2) circle (1pt) node{};
\filldraw[black] (6+.5+2/6,3/2+2/2) circle (1pt) node{};
\filldraw[black] (6+.5+3/6,3/2+3/2) circle (1pt) node{};
\end{tikzpicture}
\end{center}

Now, we specialize the points $q_{a_1},\ldots,q_{a_r}$ one by one to a fixed general point $q_a\in C_{v_0}$. Let the resulting curve be $C_0'\to X$, given by gluing together the prestable curves $(C_{v}')_{v\in \Ve(\tau)}$. From the argument in \cite[Proposition 6]{FP}, if we want to understand what happens to $C_{v_0}$ in the limit as we specialize, it suffices to understand what happens to the map $(C_{v_0},q_{a_1},\ldots,q_{a_r})\to X$ from the pointed curve $(C_{v_0},q_{a_1},\ldots,q_{a_r})$ as we specialize the points $q_{a_1},\ldots,q_{a_r}$. 

Then, $C_{v_0}$ gets replaced with the prestable curve $C_{v_0}'$ that is $C_{v_0}$ with a chain of rational curves attached at $q_a$. Proposition \ref{smooth} tells us that $C_0'\to X$ is a smooth point of $F_p$. By induction, the space of degree $e-1$ curves through $q_a$ contains only one component with curves that do not cover a line.

This means $C_0'\to X$ is in the same component of $F_p$ as the curve we get when we take $C_{v_0}\cong \mb{P}^1$ and attach a general chain of lines to $q_a$ in the same component of degree $e-1$ curves through $q_a$. This chain of lines may be a cover of a line. In this way, we have reduced to the case where $\tau$ is a chain. 

Now, let $v_0,\ldots, v_{e-1}$ be $\Ve(\tau)$, where each $v_{i}$ is connected to $v_{i+1}$ for $0\leq i\leq e-2$. Let the unique tail of $\tau$ be attached to $v_0$ and $h: C\to X$ given by $(h_{i}: C_{v_i}\to X)$ be a point of $\M(X,\tau)\cap F_p$, general in its component. As before, each $h_i: C_{v_i}\cong \mb{P}^1\to X$ is an embedding of a line. Now, we want to specialize the points of attachment. 
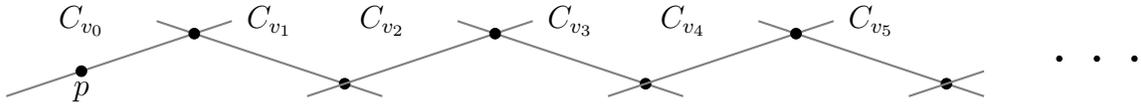
\begin{center}
\begin{tikzpicture}

\draw[gray, thick] (0,0) -- (3,1);
\draw[gray, thick] (2,1) -- (5,0);
\filldraw[black] (1,1/3) circle (2pt) node[anchor=north]{$p$} ;
\filldraw[black] (2.5,2.5/3) circle (2pt) node[anchor=west]{} ;
\filldraw[black] (1,1) circle (0pt) node{$C_{v_0}$};
\filldraw[black] (3.5,1) circle (0pt) node{$C_{v_1}$};
\filldraw[black] (4.5,.5/3) circle (2pt) node{};

\draw[gray, thick] (4,0) -- (7,1);
\draw[gray, thick] (6,1) -- (9,0);
\filldraw[black] (2.5+4,2.5/3) circle (2pt) node[anchor=west]{} ;
\filldraw[black] (1+4,1) circle (0pt) node{$C_{v_2}$};
\filldraw[black] (3.5+4,1) circle (0pt) node{$C_{v_3}$};
\filldraw[black] (4.5+4,.5/3) circle (2pt) node{};

\draw[gray, thick] (8,0) -- (11,1);
\draw[gray, thick] (10,1) -- (13,0);
\filldraw[black] (2.5+8,2.5/3) circle (2pt) node[anchor=west]{} ;
\filldraw[black] (1+8,1) circle (0pt) node{$C_{v_4}$};
\filldraw[black] (3.5+8,1) circle (0pt) node{$C_{v_5}$};
\filldraw[black] (4.5+8,.5/3) circle (2pt) node{};

\draw[gray, thick] (12,0) -- (13,1/3);
\filldraw[black] (14,1/2) circle (1pt) node{};
\filldraw[black] (14.5,1/2) circle (1pt) node{};
\filldraw[black] (15,1/2) circle (1pt) node{};
\end{tikzpicture}
\end{center}

Let $p_0=p$ and $p_i$ be the point of $C_{v_i}$ that is attached to $C_{v_{i-1}}$ under $h:C\to X$ for $1\leq i\leq e-1$. For $0\leq i\leq e-2$, let $q_i\in C_{v_i}$ be the point that is attached to $C_{v_{i+1}}$. Now, we take a 1-dimensional family that specializes $q_0$ to $p_0$. Then, we specialize $q_1$ to $p_1$. Let $C'\to X$ given by $(h_i': C_{v_i}'\to X)$ be the result after specializing $q_i$ to $p_i$ for all $0\leq i\leq n-2$. Then, each $C_{v_i}'$ for $0\leq i\leq e-2$ becomes a union of two rational curves, where the one containing $q_i$ and $p_i$ is collapsed under the map to $X$. 

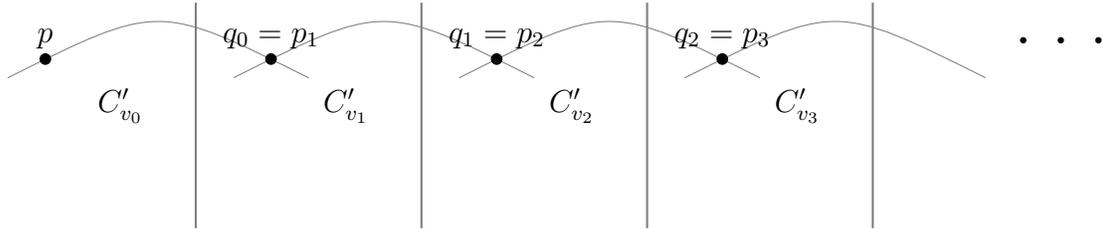
\begin{center}
\begin{tikzpicture}
\draw[gray] (0,-1) .. controls (2,0) .. (4,-1);
\draw[gray] (3,-1) .. controls (5,0) .. (7,-1);
\draw[gray] (6,-1) .. controls (8,0) .. (10,-1);
\draw[gray] (9,-1) .. controls (11,0) .. (13,-1);

\draw[gray, thick] (2.5,0) -- (2.5,-3);
\draw[gray, thick] (5.5,0) -- (5.5,-3);
\draw[gray, thick] (8.5,0) -- (8.5,-3);
\draw[gray, thick] (11.5,0) -- (11.5,-3);

\filldraw[black] (.5,-.75) circle (2pt) node[anchor=south]{$p$};

\filldraw[black] (3.5,-.75) circle (2pt) node[anchor=south]{$q_0=p_1$};
\filldraw[black] (6.5,-.75) circle (2pt) node[anchor=south]{$q_1=p_2$};;
\filldraw[black] (9.5,-.75) circle (2pt) node[anchor=south]{$q_2=p_3$};;

\filldraw[black] (1.5,-1) circle (0pt) node[anchor=north]{$C_{v_0}'$};
\filldraw[black] (4.5,-1) circle (0pt) node[anchor=north]{$C_{v_1}'$};
\filldraw[black] (7.5,-1) circle (0pt) node[anchor=north]{$C_{v_2}'$};
\filldraw[black] (10.5,-1) circle (0pt) node[anchor=north]{$C_{v_3}'$};

\filldraw[black] (13.5,-1/2) circle (1pt) node{};
\filldraw[black] (14,-1/2) circle (1pt) node{};
\filldraw[black] (14.5,-1/2) circle (1pt) node{};
\end{tikzpicture}
\end{center}

Proposition \ref{smooth} tells us that $C'\to X$ is a smooth point of $F_p$. However, we note that $C'\to X$ is also a specialization of a strict map $D\to X$ lying in $\mathcal{M}(X,\tau_{{\rm comb}})\cap F_p$, where $\Ve(\tau_{{\rm comb}})=\{v_{{\rm center}},v_0,\ldots,v_{e-1}\}$ with $\beta(v_{{\rm center}})=0$ and $\beta(v_i)=1$, and the edges of $\tau_{{\rm comb}}$ connect $v_{{\rm center}}$ to each of $v_0,\ldots,v_{e-1}$. The unique tail of $\tau_{{\rm comb}}$ is attached to $v_{{\rm center}}$. Proposition \ref{smooth} says that a general choice of a strict map $D$ given by $(g_v: D_v\to X)_{v\in \Ve(\tau_{{\rm comb}})}$ is a smooth point of $F_p$. In fact, since the choice of the maps $D_{v_i}\to X$ is discrete, every choice of a strict map $D\to X$ is smooth. 

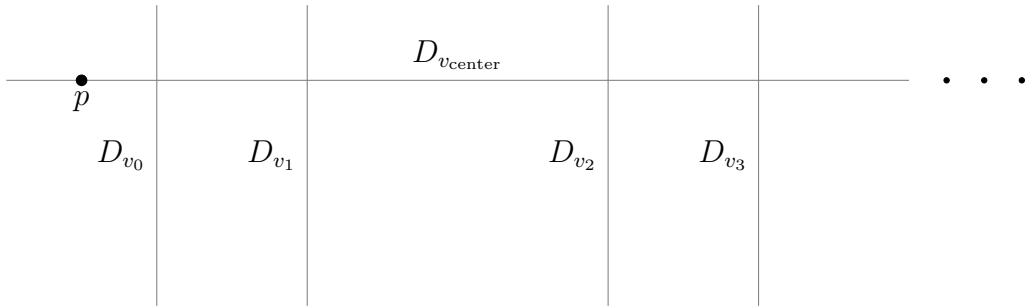
\begin{center}
\begin{tikzpicture}
\draw[gray] (0,0) --(12,0);

\filldraw[black] (6,0) circle (0pt) node[anchor=south]{$D_{v_{{\rm center}}}$};
\filldraw[black] (1,0) circle (2pt) node[anchor=north]{$p$};

\draw[gray] (2,1) --(2,-3);
\filldraw[black] (2,-1) circle (0pt) node[anchor=east]{$D_{v_{0}}$};

\draw[gray] (4,1) --(4,-3);
\filldraw[black] (4,-1) circle (0pt) node[anchor=east]{$D_{v_{1}}$};

\draw[gray] (8,1) --(8,-3);
\filldraw[black] (8,-1) circle (0pt) node[anchor=east]{$D_{v_{2}}$};

\draw[gray] (10,1) --(10,-3);
\filldraw[black] (10,-1) circle (0pt) node[anchor=east]{$D_{v_{3}}$};

\filldraw[black] (12.5,0) circle (1pt) node{};
\filldraw[black] (13,0) circle (1pt) node{};
\filldraw[black] (13.5,0) circle (1pt) node{};
\end{tikzpicture}
\end{center}

The backbone $D_{v_{{\rm center}}}$ gets collapsed to $p$ under the map to $X$ and each $D_{v_i}$ gets mapped to a line through $p$. To finish, we need to show that all strict maps $D\to X$, where not all of the $D_{v_i}$ get mapped to the same line, are in the same component of $F_p$. 

Let $S\subset \{v_0,\ldots,v_{e-1}\}$ be a strict subset, with $\# S\geq 2$ and such that the maps $D_v\to X$ for $v\in S$ do not all embed as the same line. Then, we can specialize $D\to X$ to $D'\to X$, where $D_{v_{\rm center}}'$ is now a rational curve with two components, one that is attached to $D_{v}'$ for all $v\in S$ and the other that contains $p$ and is attached to $D_{v}'$ for $v\notin S$. The maps $D_{v_i}\cong D_{v_i}'\to X$ are unchanged. Then, using the induction hypothesis, this map $D'\to X$ is in the same component of $F_p$ as the map we would get by modifying the lines $D_v'\to X$ for $v\in S$, so long as $D_v'\to X$ for $v\in S$ do not all embed as the same line. 

The end result is that $D\to V$ is in the same component as the map we would get by modifying $D_v\to X$ for $v\in S$, so long as we maintain $D_v\to X$ for $v\in S$ do not all embed as the same line. In this way, we see that all strict maps $D\to X$, where not all of the $D_{v_i}$ get mapped to the same line, are in the same component of $F_p$.
\end{proof}
\begin{center}
\begin{figure}
\label{pagefigure}
\begin{tikzpicture}
\draw (-1,0) node{Start: };

\draw[gray, thick] (0,0) -- (3,1);
\draw[gray, thick] (2,1) -- (5,0);
\filldraw[black] (1,1/3) circle (2pt) node[anchor=north]{$p$} ;

\draw[gray, thick] (4,0) -- (7,1);

\draw[gray, thick] (10,2/3)--(14,2/3);
\filldraw[black] (11,2/3) circle (2pt) node[anchor=north]{$p$} ;
\draw[gray, thick] (12,1)--(12,-1/3);
\draw[gray, thick] (13,1)--(13,-1/3);

\draw[gray, thick] (10,-4)--(14,-4);
\filldraw[black] (11,-4) circle (2pt) node[anchor=north]{$p$} ;
\draw[gray, dashed, thick] (12,-4+1/3)--(12,-6);
\draw[gray, thick] (12-1/3,-4-2/3)--(13,-4-2/3);
\draw[gray, thick] (12-1/3,-5-2/3)--(13,-5-2/3);

\draw[->,black, thick] (11.5,-2/3)--(11.5, -4+1/3);
\draw (11,-2) node[rotate=90]{specialization};

\draw[->, black, thick] (5,-1)--(10,-4+1/3);
\draw (6,-2) node[rotate=-27]{induction};

\draw[gray, dashed] (0,-1-4) .. controls (2,0-4) .. (4,-1-4);
\draw[gray, dashed] (3,-1-4) .. controls (5,0-4) .. (7,-1-4);

\draw[gray, thick] (2.5,0-4) -- (2.5,-3-4);
\draw[gray, thick] (5.5,0-4) -- (5.5,-3-4);
\draw[gray, thick] (6,-.75-4) -- (9,-.75-4);

\filldraw[black] (.5,-.75-4) circle (2pt) node[anchor=south]{$p$};

\draw[->,black, thick] (4,-2/3)--(4, -4+1/3);
\draw (3.5,-2) node[rotate=90]{specialization};

\draw[gray, dashed] (0,-1-4-6) -- (7,-1-4-6);

\draw[gray, thick] (2.5,0-4-6) -- (2.5,-3-4-6);
\draw[gray, thick] (4,0-4-6) -- (4,-3-4-6);
\draw[gray, thick] (5.5,0-4-6) -- (5.5,-3-4-6);

\filldraw[black] (.5,-1-4-6) circle (2pt) node[anchor=south]{$p$};
\draw (3,-2-4-6) node[rotate=0]{$L_1$};
\draw (4.5,-2-4-6) node[rotate=0]{$L_2$};
\draw (6,-2-4-6) node[rotate=0]{$L_2$};

\draw[->,black, thick] (4,-2/3-4-5)--(4, -6);
\draw (3.5,-8) node[rotate=90]{specialization};

\draw[gray, dashed] (0+7.5,-1-4-6) .. controls (2+7.5,0-4-6) .. (4+7.5,-1-4-6);
\draw[gray, dashed] (3+7.5,-1-4-6) .. controls (5+7.5,0-4-6) .. (7+7.5,-1-4-6);

\draw[gray, thick] (8.5,0-4-6) -- (8.5,-3-4-6);
\draw[gray, thick] (10,0-4-6) -- (10,-3-4-6);
\draw[gray, thick] (12.5,0-4-6) -- (12.5,-3-4-6);

\filldraw[black] (.5+7.5,-.75-4-6) circle (2pt) node[anchor=south]{$p$};
\draw (9,-2-4-6) node[rotate=0]{$L_1$};
\draw (10.5,-2-4-6) node[rotate=0]{$L_2$};
\draw (13,-2-4-6) node[rotate=0]{$L_2$};

\draw[->,black, thick] (6,-10)--(8,-10);
\draw (7,-9.5) node[rotate=0]{specialization};

\draw[gray, dashed] (0+7.5,-1-4-6-5) .. controls (2+7.5,0-4-6-5) .. (4+7.5,-1-4-6-5);
\draw[gray, dashed] (3+7.5,-1-4-6-5) .. controls (5+7.5,0-4-6-5) .. (7+7.5,-1-4-6-5);

\draw[gray, thick] (8.5,0-4-6-5) -- (8.5,-3-4-6-5);
\draw[gray, thick] (10,0-4-6-5) -- (10,-3-4-6-5);
\draw[gray, thick] (12.5,0-4-6-5) -- (12.5,-3-4-6-5);

\filldraw[black] (.5+7.5,-.75-4-6-5) circle (2pt) node[anchor=south]{$p$};
\draw (9,-2-4-6-5) node[rotate=0]{$L_1'$};
\draw (10.5,-2-4-6-5) node[rotate=0]{$L_2'$};
\draw (13,-2-4-6-5) node[rotate=0]{$L_2$};

\draw[->, black, thick] (11,-.75-4-6-4.5)--(11,-.75-4-7.5);
\draw[<-, black, thick] (11,-.75-4-6-4.5)--(11,-.75-4-7.5);
\draw (10.5,-.75-4-6-3) node[rotate=90]{induction};

\draw[gray, dashed] (0,-1-4-6-5) -- (7,-1-4-6-5);

\draw[gray, thick] (2.5,0-4-6-5) -- (2.5,-3-4-6-5);
\draw[gray, thick] (4,0-4-6-5) -- (4,-3-4-6-5);
\draw[gray, thick] (5.5,0-4-6-5) -- (5.5,-3-4-6-5);

\filldraw[black] (.5,-1-4-6-5) circle (2pt) node[anchor=south]{$p$};
\draw (3,-2-4-6-5) node[rotate=0]{$L_1'$};
\draw (4.5,-2-4-6-5) node[rotate=0]{$L_2'$};
\draw (6,-2-4-6-5) node[rotate=0]{$L_2$};

\draw[->,black, thick] (6,-10-5)--(8,-10-5);
\draw (7,-9.5-5) node[rotate=0]{specialization};
\end{tikzpicture}
\caption{A diagram of the proof of Theorem \ref{fiber} in the case $e=3$. Components depicted with solid lines embed as lines in the hypersurface $X$ and components depicted with dashed lines are collapsed. The lines $L_1,L_2,L_3$ are not all the same by assumption, so we assume without loss of generality that $L_1\neq L_2$. 
}
\end{figure}
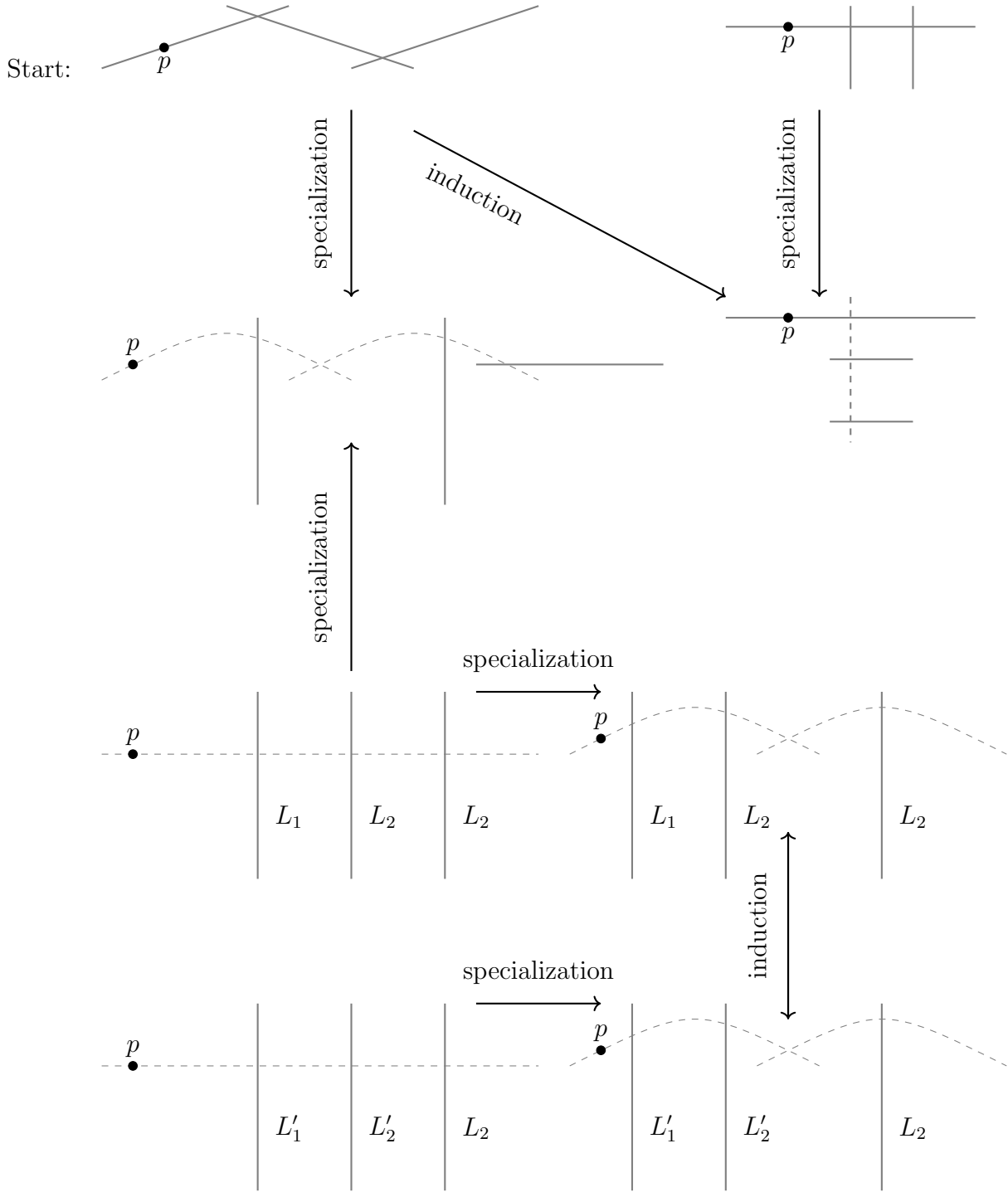
\end{center}
\clearpage
\bibliographystyle{plain}
\bibliography{references.bib}
\end{document}